\documentclass[draft,english]{amsart}

\usepackage{amsthm}
\usepackage{amssymb}
\usepackage{amsmath}
\usepackage{bbm}
\usepackage{mathrsfs}
\usepackage{xypic}
\usepackage{textcomp}

\newtheorem{thm}{Theorem}

\newtheorem{dfn}[thm]{Definition}
\newtheorem{rem}[thm]{Remark}
\newtheorem{rems}[thm]{Remarks}
\newtheorem{exa}[thm]{Example}
\newtheorem{exas}[thm]{Examples}
\newtheorem{prop}[thm]{Proposition}

\newcommand{\Pb}{\mathop{\text{\bf P}}\nolimits} 
\newcommand{\Gl}{\mathop{\text{\rm GL}}\nolimits}
\newcommand{\calR}{K[X_1,\ldots,X_n]}
\newcommand{\calD}{K\left[\frac{\partial }{\partial X_1},\ldots,\frac{\partial}{\partial X_n}\right]}

\title{Computing invariants of cubic surfaces}


\author{Andreas-Stephan Elsenhans}
\address{Andreas-Stephan Elsenhans\\
Institut f\"ur Mathematik\\
Universit\"at W\"urzburg}
\email{stephan.elsenhans@mathematik.uni-wuerzburg.de}

\author{J\"org Jahnel}
\address{J\"org Jahnel\\ Department Mathematik\\ Universit\"at Siegen}
\email{jahnel@mathematik.uni-siegen.de}

\begin{document}

\maketitle

\begin{abstract}
We report on the computation of invariants,
covariants, and contravariants of cubic surfaces.
All algorithms are implemented in the computer
algebra system {\tt magma}.
\end{abstract}

\section{Introduction}

Given two hypersurfaces of the same degree in projective space over an algebraically closed 
field, one may ask for the existence of an automorphism of the projective space that maps one 
of the hypersurfaces to the other.
It turns out that if the hypersurfaces are stable~\cite[Def.~1.7]{MFK} in the sense of geometric invariant 
theory such an isomorphism exists 
if and only if all the invariants of the hypersurfaces coincide~\cite{Mu}.

Aside from cubic curves in $\Pb^2$ and quartic surfaces in $\Pb^3$, an isomorphism between 
smooth hypersurfaces of degree $d \geq 3$ always extends to an automorphism of the 
ambient projective space~\cite[Th.~2]{MM}.
Thus, the invariants may be used to test abstract isomorphy.

If the base field is not algebraically closed, two varieties
with equal invariants can differ by a twist.
A necessary condition for the existence of a non-trivial twist
is that the variety has a non-trivial automorphism.

In this article, we focus on the case of cubic surfaces. 
For them, it was proven by Clebsch~\cite{Cl} that the ring 
of invariants of even weight is generated by five invariants of 
degrees 8, 16, 24, 32, and 40. Later, Salmon~\cite{Sa3} 
worked out explicit formulas for these invariants based on the 
pentahedral representation of the cubic surface, 
introduced by Sylvester~\cite{Sy}. Using modern computer algebra, it is possible
to compute the pentahedral representation of a given
cubic surface and to deduce the invariants from this~\cite{EJ1}.

We describe a different approach to compute the Clebsch-Salmon
invariants, linear covariants, and some contravariants of cubic 
surfaces,  based on the Clebsch transfer principle. Using this, we also 
compute an  invariant of degree 100~\cite[Sec.~9.4.5]{Do} and odd weight that vanishes
if and only if the cubic surface has a non-trivial automorphism.
The square of this invariant is a polynomial expression in Clebsch's
invariants.

This answers the question of isomorphy for all
stable cubic surfaces over algebraically closed fields 
and for all surfaces over non-closed fields, for which the 
degree 100 invariant does not vanish.  

All algorithms are implemented in the computer algebra
system {\tt magma}~\cite{BCP}.

\section{The Clebsch-Salmon invariants}

\begin{dfn} \label{def_inv}
Let $K$ be a field of characteristic zero and $K[X_1,\ldots,X_n]^{(d)}$ the $K$-vector space of all 
homogeneous forms of degree $d$. 
Further, we fix the left group action 
\begin{eqnarray*}
\Gl_n(K) \times K[X_1,\ldots,X_n]  \rightarrow K[X_1,\ldots,X_n],\quad (M,f) \mapsto M \cdot f,
\end{eqnarray*}
with $(M \cdot f)(X_1,\ldots,X_n) := f((X_1,\ldots,X_n) \, M)$.
Finally, on the polynomial ring $K[Y_1,\ldots,Y_n]$, we choose the action
\begin{eqnarray*}
\Gl_n(K) \times K[Y_1,\ldots,Y_n]  \rightarrow K[Y_1,\ldots,Y_n], \quad
(M,f) \mapsto M \cdot f,
\end{eqnarray*}
given by $(M \cdot f)(Y_1,\ldots,Y_n) := f((Y_1,\ldots,Y_n) \left(M^{-1}\right)^\top)$.
\begin{enumerate}
\item
An {\em invariant $I$ of degree $D$ and weight $w$} is a map $K[X_1,\ldots,X_n]^{(d)} \rightarrow K$ that may be
given by a homogeneous polynomial of degree $D$ in the coefficients of $f$ and  satisfies
$$
I(M \cdot f) =  \det(M)^w \cdot I(f), 
$$ 
for all $M \in \Gl_n(K)$  and all forms $f \in K[X_1,\ldots,X_n]^{(d)}$.
\item
A {\em covariant $C$ of degree $D$, order $p$, and weight $w$} is a map
$$
K[X_1,\ldots,X_n]^{(d)} \rightarrow K[X_1,\ldots,X_n]^{(p)}
$$
such that each coefficient of $C(f)$ is a homogeneous degree $D$ polynomial 
in the coefficients of $f$ and that satisfies
$$
C(M \cdot f) =  \det(M)^w \cdot M \cdot (C(f)),
$$ 
for all $M \in \Gl_n(K)$ and all forms $f \in K[X_1,\ldots,X_n]^{(d)}$.
\item
A {\em contravariant $c$ of degree $D$, order $p$, and weight $w$} is a map 
$$
K[X_1,\ldots,X_n]^{(d)} \rightarrow K[Y_1,\ldots,Y_n]^{(p)}
$$
such that each coefficient of $c(f)$ is a homogeneous degree $D$ polynomial 
in the coefficients of $f$ and that satisfies
$$
c(M \cdot f) =  \det(M)^w \cdot M \cdot c(f),
$$ 
for all $M \in \Gl_n(K)$ and all forms $f \in K[X_1,\ldots,X_n]^{(d)}$. Note that the right hand
side uses the action on $K[Y_1,\ldots,Y_n]$.
\end{enumerate}
\end{dfn}

\begin{rems}
\begin{enumerate}
\item
The set of all invariants is a commutative ring and an algebra over the base field.
\item
The set of all covariants (resp. contravariants) is a commutative ring and a module over the ring of invariants.
\item
Geometrically, the vanishing locus of $f$ or a  covariant $C(f)$ is a subset of the projective space whereas
the vanishing locus of a contravariant $c(f)$ is a subset of the dual projective space. Replacing the matrix by
the transpose inverse matrix gives the action on the dual space in a naive way.
\end{enumerate}
\end{rems}

\begin{exas}
\begin{enumerate}
\item
The discriminant of binary forms of degree $d$ is an invariant of degree $2d - 2$ and weight 
$d(d-1)$~\cite[Chap.~2]{Ol}.
\item
Let $f$ be a form  of degree $d > 2$ in $n$ variables. Then the {\it Hessian}~$H$
defined by
$$
H(f) := \det \left(\frac{\partial^2 f}{\partial X_i \, \partial X_j} \right)_{i,j =1,\ldots,n}
$$
is a covariant of degree $n$, order $(d-2) n$, and weight $2$.
\item
Let a smooth plane curve $V \subset \Pb^2$ be given by a ternary form $f$ of degree $d$. 
Mapping $f$ to the form that defines the dual curve~\cite[Sec.~1.2.2]{Do} of $V$ is an example of a contravariant 
of degree $2d - 2$ and order $d(d-1)$.
\end{enumerate}
\end{exas}

\subsection*{Salmon's formulas}
A cubic surface given by a system of equations of the shape
$$
a_0 X_0^3 + a_1 X_1^3 +a_2 X_2^3 +a_3 X_3^3 +a_4 X_4^3 = 0 
, \quad X_0 + X_1 + X_2 + X_3 + X_4 = 0
$$
is said to be in {\it pentahedral form}. The coefficients $a_0,\ldots,a_4$ are called the
pentahedral coefficients of the surface. The cubic surfaces that have a pentahedral form
are a Zariski open subset in the Hilbert scheme of all cubic surfaces. Thus, it suffices to give
the invariants for these surfaces.
For this, we denote by $\sigma_1,\ldots,\sigma_5$ the elementary symmetric functions 
of the pentahedral coefficients. Then the Clebsch-Salmon invariants 
(as mentioned in the introduction) 
of the cubic surface are given by~\cite[\S~467]{Sa3},
\begin{eqnarray*}
I_8 = \sigma_4^2 - 4 \sigma_3 \sigma_5, \quad
I_{16} = \sigma_1 \sigma_5^3, \quad
I_{24} = \sigma_4 \sigma_5^4, \quad
I_{32} = \sigma_2 \sigma_5^6, \quad
I_{40} = \sigma_5^8\, .
\end{eqnarray*}
Further, Salmon lists four linear covariants of degrees 11, 19, 27, and 43~\cite[\S~468]{Sa3} 
\begin{align*}
L_{11} &= \sigma_5^2 \sum_{i=0}^4 a_i x_i, &
L_{19} &=  \sigma_5^4 \sum_{i=0}^4 \frac{1}{a_i} x_i, \\
L_{27} &=  \sigma_5^5 \sum_{i=0}^4 a_i^2 x_i, &  
L_{43} &=  \sigma_5^8 \sum_{i=0}^4 a_i^3 x_i \,.
\end{align*}
Finally, the $4 \times 4$-determinant of the matrix formed by the coefficients of these 
linear covariants of a cubic surface in $\Pb^3$ is an invariant $I_{100}$ of degree 100.
It vanishes if and only if the surface has Eckardt points or equivalently a non-trivial 
automorphism group~\cite[Sec.~9.4.5, Table~9.6]{Do}. The square of $I_{100}$ can be expressed in terms of the other invariants above.
For a modern view on these invariants, we refer to~\cite[Sec.~9.4.5]{Do}.


\section{Transvection}
One classical approach to write down invariants is to 
use the transvection (called \"Uberschiebung in German). 
This is part of the so called symbolic method~\cite[Chap.~8,~\S2]{We},~\cite[App.~B.2]{Hu}.
We illustrate it in the case of ternary forms.

\begin{dfn}
Let $K[X_1,\ldots,X_n,Y_1,\ldots,Y_n,Z_1,\ldots,Z_n]$ be the polynomial ring in $3 n$
variables. For $i,j,k \in \{1,\ldots,n\}$, we denote by $(i\, j\, k)$ the differential operator 
\begin{eqnarray*}
(i\, j\, k) := 
\det
\left(
\begin{array}{ccc}
\frac{\partial}{\partial X_i} &  
\frac{\partial}{\partial X_j} &  
\frac{\partial}{\partial X_k} \\
\frac{\partial}{\partial Y_i} &  
\frac{\partial}{\partial Y_j} &  
\frac{\partial}{\partial Y_k} \\
\frac{\partial}{\partial Z_i} &  
\frac{\partial}{\partial Z_j} &  
\frac{\partial}{\partial Z_k} \\
\end{array}
\right) 
\, .
\end{eqnarray*}
\end{dfn}

\begin{exa} 
Using this notation, the {\em Aronhold invariants} $S$ and $T$
of the ternary cubic form $f$ are given by
\begin{align*}
S(f) &:=
(1\, 2 \, 3) (2 \, 3 \, 4) (3 \, 4 \, 1) (4 \, 1 \, 2) 
f(X_1,Y_1,Z_1) \cdots f(X_4,Y_4,Z_4), \\
T(f) &:= 
(1 \, 2 \, 3) (1 \, 2 \, 4) (2 \, 3 \, 5) (3 \, 1 \, 6) 
(4 \, 5 \, 6)^2 f(X_1,Y_1,Z_1) \cdots f(X_6,Y_6,Z_6)\, .
\end{align*} 
The first one is of degree and weight~$4$, the second one is of degree and weight~$6$.
Using $S$ and $T$, one can write down the discriminant of a 
ternary cubic as $\Delta := S^3 - 6 T^2$. The discriminant vanishes if 
and only if the corresponding cubic curve is singular.

See~\cite[Sec.~V]{Sa2} for a historical and~\cite[Sec.~3.4.1]{Do} 
for modern references concerning invariants of ternary
cubic forms. 
\end{exa}

\begin{rem}
One can use the transvection to write down invariants of quaternary forms, as well.
For example, if $f$ is a quartic form in four variables then
$$
(1\, 2\, 3\, 4)^4 f(X_1,Y_1,Z_1,W_1) \cdots f(X_4,Y_4,Z_4,W_4)
$$ 
is an invariant of degree 4. Here, $(1\, 2\, 3 \, 4)$ denotes the  
differential operator 
$$
(1\, 2\, 3 \, 4) := 
\det
\left(
\begin{array}{cccc}
\frac{\partial}{\partial X_1} &  
\frac{\partial}{\partial X_2} &  
\frac{\partial}{\partial X_3} &
\frac{\partial}{\partial X_4} \\
\frac{\partial}{\partial Y_1} &  
\frac{\partial}{\partial Y_2} &  
\frac{\partial}{\partial Y_3} &
\frac{\partial}{\partial Y_4} \\
\frac{\partial}{\partial Z_1} &  
\frac{\partial}{\partial Z_2} &  
\frac{\partial}{\partial Z_3} &
\frac{\partial}{\partial Z_4} \\
\frac{\partial}{\partial W_1} &  
\frac{\partial}{\partial W_2} &  
\frac{\partial}{\partial W_3} &
\frac{\partial}{\partial W_4} 
\end{array}
\right)\, .
$$
For a quaternary cubic form, one can apply this 
to its Hessian to get an invariant of degree 16.
However, a direct evaluation of such formulas for
forms in four variables is too slow in practice. The reason is that  both the 
differential operators and the product 
$f(X_1,Y_1,Z_1,W_1) \cdots f(X_4,Y_4,Z_4,W_4)$ 
usually have many terms.
\end{rem}

\section{The Clebsch transfer principle}
We refer to~\cite[Sec.~3.4.2]{Do} for a detailed and modern description of
the Clebsch transfer principle. The basic idea is to
compute a contravariant of a form of degree $d$ in $n$
variables out of an invariant of a form of degree $d$ 
in $(n-1)$ variables. 

\begin{dfn} 
\begin{enumerate}
\item
We consider the vector space $V = K^n$ and choose the volume form given by the
determinant. We have the following isomorphism
$$
\Phi \colon \Lambda^{n-1} V \rightarrow V^*,\quad 
v_1 \wedge \dots \wedge v_{n-1} \mapsto 
(v \mapsto \det (v,v_1,\ldots,v_{n-1}))\,.
$$
\item
Let $I$ be a degree $D$, weight $w$ invariant on $K[U_1,\ldots,U_{n-1}]^{(d)}$. 
Then the {\it Clebsch transfer} of $I$ is the 
contravariant $\tilde{I}$ of degree $D$ and order $w$ 
$$
\tilde{I} \colon K[X_1,\ldots,X_{n}]^{(d)} \rightarrow K[Y_1,\ldots,Y_n]^{(w)},
$$ 
given by
$$
\tilde{I}(f) \colon (K^n)^* \rightarrow K,\quad 
l \mapsto I(f(U_1 v_1 + \cdots + U_{n-1} v_{n-1}))\, .
$$ 
Here, $v_1,\ldots,v_{n-1}$ are given by
$v_1 \wedge \ldots \wedge v_{n-1} = \Phi^{-1}(l)$.
Note that $\tilde{I}(f)$, as defined, is indeed a polynomial mapping
and homogeneous of degree~$w$.
\end{enumerate}
\end{dfn}

\begin{exa}
Denote by $S$ and $T$ the invariants of ternary cubic forms,
introduced above.  
Then $\tilde{S}$ is a degree 4, order 4 contravariant of quaternary cubic forms. 
Further, $\tilde{T}$ is a contravariant of degree 6 and order 6.

The discriminant of a cubic curve is given by $\Delta = S^3 - 6 T^2$. It vanishes if and 
only if the cubic curve is singular.
Thus, the dual surface of the smooth cubic surface $V(f)$
is given by $\tilde{\Delta}(f) = \tilde{S}(f)^3 - 6 \tilde{T}(f)^2 = 0$.
\end{exa}

\begin{rem}
By definition, the dual surface of a smooth surface $V(f) \subset \Pb^3$ is the set of all 
tangent hyperplanes of $V(f)$. A plane $P \in (\Pb^3)^*$ is tangent if and only it the 
intersection $V(f) \cap P$ is singular. Thus, $P$ is a point on the dual surface if and only if
$\tilde{\Delta}(f)(P) = 0$. Here, $\Delta$ is the discriminant of ternary forms of the same
degree as $f$.
\end{rem}

\begin{rem}
For a given cubic form $f \in K[X,Y,Z,W]$, we compute $\tilde{S}(f)$ 
by interpolation as follows:
\begin{enumerate}
\item
Choose 35 vectors $p_1,\ldots,p_{35} \in \left(K^4\right)^*$ in general position.
\item
Compute $\Phi^{-1}(p_i)$, for $i = 1,\ldots,35$.
\item
Compute $s_i := S(f(U_1  v_1 + U_2 v_2 + U_3 v_3))$, for $v_1 \wedge v_2 \wedge v_3 =  \Phi^{-1}(p_i)$
and all $i = 1,\ldots,35$.
\item
Compute the degree $4$ form $\tilde{S}(f)$ by interpolating the arguments $p_i$ and the values $s_i$. 
\end{enumerate}
We can compute $\tilde{T}(f)$ in the same way. The only modification necessary is to
increase the number of vectors, as the space of sextic forms is of dimension 84.
\end{rem}

\section{Action of contravariants on covariants and vice versa}

\begin{enumerate}
\item
Recall that the rings $K[X_1,\ldots, X_n]$ and $K[Y_1,\ldots, Y_n]$ are equipped with
 $\Gl_n(K)$-actions, as introduced in Definition~\ref{def_inv}.
\item
The ring of differential operators
$$K\left[\frac{\partial }{\partial X_1},\ldots,\frac{\partial }{\partial X_n}\right]$$ 
acts on $\calR$.
\item
The $\Gl_n(K)$-action on $\calD$ given by 
$$
M \cdot \left(\frac{\partial }{\partial v} \right) := 
\frac{\partial }{\partial (v \cdot M^{-1})} \mbox{ for all } v \in K^n
$$
results in the equality
$$
M \cdot \left(\frac{\partial f}{\partial v}\right) = 
\left( M \cdot \frac{\partial}{\partial v}  \right) \left(M \cdot f \right),
$$
for all $f \in K[X_1,\ldots,X_n]$ and all $v \in K^n$.
\item
The map 
$$
\psi 
\colon K[Y_1,\ldots,Y_n] \rightarrow 
K\left[\frac{\partial }{\partial X_1},\ldots,\frac{\partial }{\partial X_n}\right], \quad
Y_i \mapsto \frac{\partial }{\partial X_i}
$$
is an isomorphism of rings. Further, for each $M \in \Gl_n(K)$,
we have the following commutative diagram
\begin{eqnarray*}
\diagram
K[Y_1,\ldots,Y_n] \rrto^{\psi~~~} \dto_M & & 
\calD 
 \dto^{M} \\
K[Y_1,\ldots,Y_n] \rrto^{\psi~~~}             & & 
\calD 
\enddiagram_{\displaystyle .}
\end{eqnarray*}
In other words, $\psi$ is an isomorphism of $\Gl_n(K)$-modules.
\item
Let $C$ be a covariant and $c$ a contravariant on $K[X_1,\ldots,X_n]^{(d)}$. 
Denote the order of $C$ by $P$ and the order of $c$ by $p$.
For $P \geq p$, we define 
\begin{eqnarray*}
&c \vdash C \colon  K[X_1,\ldots,X_n]^{(d)} \rightarrow K[X_1,\ldots,X_n]^{(P-p)},\quad
f \mapsto   \psi(c(f)) \left(C(f)\right)\, .
\end{eqnarray*} 
The notation $\vdash$ follows~\cite[p.~304]{Hu}.
\item
Assume $c \vdash C$ not to be zero.
If $p < P$ then $c \vdash C$ is a covariant of order $P - p$. 
If $p = P$ then $c \vdash C$ is an invariant.
In both cases, the degree of  $c \vdash C$ is the sum of the degrees of $c$ and $C$. 
\item
Similarly to $\psi$, one can introduce a map 
$$\widehat{\psi} \colon K[X_1,\ldots,X_n] \rightarrow
K\left[\frac{\partial }{\partial Y_1},\ldots,\frac{\partial }{\partial Y_n} \right],\quad
X_i \mapsto \frac{\partial }{\partial Y_i}\,.
$$
As above, $\widehat{\psi}$ is an isomorphism of rings and $\Gl_n(K)$-modules. 
Let $C$ a covariant and  $c$ a contravariant 
 on $K[X_1,\ldots,X_n]^{(d)}$.
We  define $C \vdash c$ by
$$
(C \vdash c)(f) := \widehat{\psi}(C(f)) \left(c(f)\right)\, .
$$ 
\item
Assume $c \vdash C$ not to be zero.
If $p > P$ then $C \vdash c$ is a contravariant of order $p - P$. 
If $p = P$ then $C \vdash c$ is an invariant.
In both cases, the degree of  $C \vdash c$ is the sum of the degrees of $C$ and $c$. 
\end{enumerate}

\section{Explicit invariants of cubic surfaces}

\begin{rems}
\begin{enumerate}
\item
It is well known that the ring of invariants of quaternary
cubic forms is generated by the six  invariants of degrees
8, 16, 24, 32, 40, and 100~\cite[Sec.~9.4.5]{Do}. 
The first five generators are primary invariants~\cite[Def.~2.4.6]{DK}.
Thus, the vector spaces of all 
invariants of degrees 8, 16, 24, 32 and 40  are of 
dimensions 1, 2, 3, 5, and 7.   
In general, these dimensions are encoded in the 
Molien series, which can be computed efficiently using 
character theory~\cite[Ch.~4.6]{DK}.
\item
In the lucky case that one is able to write down a basis
of the vector space of all invariants of a given degree $d$,
one can find an expression of a given invariant of degree $d$
by linear algebra. This requires that the invariant is known
for sufficiently many surfaces. 
For cubic surfaces, this is provided by the pentahedral equation.
\item
Applying the methods above, we can write down many invariants for quaternary cubic forms.
We start with the form $f$, its Hessian covariant $H(f)$, and the contravariant $\tilde{S}(f)$.
Then we apply known covariants to contravariants and vice versa. Further, one can
multiply two covariants or contravariants to get a new one. 
For efficiency, it is useful to keep the orders of the covariants and contravariants as small
as possible. This way, they will not consist of too many terms.
\end{enumerate}
\end{rems}

\begin{prop}
Let $f$ be a quarternary cubic form. With
\begin{align*}
C_{4,0,4} &:= \tilde{S}(f), 
& C_{4,4} &:= H(f), \\
C_{6,2} &:= C_{4,0,4} \vdash f^2, 
& C_{9,3} &:=  C_{4,0,4}  \vdash (f \cdot C_{4,4}), \\
C_{10,0,2} &:= C_{6,2} \vdash C_{4,0,4}, 
& C_{11,1a} &:= C_{10,0,2} \vdash f, \\
C_{13,0,1} &:= C_{9,3} \vdash C_{4,0,4}, 
& C_{14,2} &:= C_{10,0,2} \vdash  C_{4,4}, \\ 
C_{14,2a} &:=  C_{13,0,1} \vdash f, 
& C_{19,1a} &:= C_{13,0,1} \vdash C_{6,2}, 
\end{align*}
the following expressions
\begin{align*}
I_8 &:= \frac{1}{2^{11} \cdot 3^9}  C_{4,0,4} \vdash C_{4,4},\\
I_{16} &:= \frac{1}{2^{30} \cdot 3^{22}} C_{6,2} \vdash C_{10,0,2}, \\
I_{24} &:= \frac{1}{2^{41} \cdot 3^{33}} C_{10,0,2} \vdash C_{14,2}, \\
I_{32a} &:= C_{10,0,2} \vdash C_{11,1a}^2, \\     
I_{32} &:= \frac{2}{5}(I_{16}^2 - \frac{1}{2^{60} \cdot 3^{44}} \cdot  I_{32a}), \\
I_{40a} &:= C_{4,0,4} \vdash (C_{11,1a}^2 \cdot C_{14,2}), \\
I_{40} &:= \frac{-1}{100} \cdot I_8 \cdot I_{32} - \frac{1}{50} \cdot I_{16} \cdot I_{24} 
        - \frac{1}{2^{72} \cdot 3^{53} \cdot 5^2} I_{40a},
\end{align*}
give the  Clebsch-Salmon invariants $I_8,\ I_{16},\ I_{24},\ I_{32},$ and  $I_{40}$.
Further, with
\begin{align*}
 C_{11,1} :=&  \frac{1}{2^{20} 3^{15}} C_{11,1a}, \\
 C_{19,1} :=& \frac{1}{2^{33} \cdot 3^{24} \cdot 5} (C_{19,1a} + 2^{32} \cdot 3^{24} \cdot I_8 \cdot C_{11,1a}),  \\ 
 C_{27,1a} :=& \frac{1}{2^{42} 3^{33}} C_{13,0,1} \vdash C_{14,2a}, \\ 
 C_{27,1} :=& I_{16} \cdot C_{11,1} + \frac{1}{200}(C_{27,1a} - 2 \cdot I_8^2 \cdot C_{11,1} - 10 \cdot I_8 \cdot C_{19,1}), \\
 C_{43,1a} :=& \frac{1}{2^{68} \cdot 3^{53}} C_{13,0,1} \vdash ( C_{13,0,1} \vdash (C_{13,0,1} \vdash C_{4,4})), \\
 C_{43,1} :=& \frac{-1}{1000}  C_{43,1a} - \frac{1}{200} \cdot I_8^2 \cdot C_{27,1} + I_{16} \cdot C_{27,1} \\
           & + \frac{1}{1000} \cdot I_8^3 \cdot C_{19,1} -\frac{1}{10} \cdot I_8 \cdot I_{16} \cdot C_{19,1} - I_{24} \cdot C_{19,1} \\
          & + \frac{1}{200} \cdot I_8^2 \cdot I_{16}  \cdot C_{11,1} + \frac{3}{20}  \cdot I_8 \cdot I_{24} \cdot C_{11,1},
\end{align*}
$C_{11,1},\ C_{19,1},\ C_{27,1},$ and $C_{43,1}$ are Salmon's linear covariants.
Here, we use the first index to indicate the degree of an invariant, covariant, or contravariant. The second 
index is the order of a covariant, whereas the third index is the order of a contravariant.
Finally, we can compute $I_{100}$ as the determinant of the 4 linear covariants.
\end{prop}

\begin{proof}
The following {\tt magma} script shows in approximately one second of CPU time that the algorithm 
as described above coincides with Salmon's formulas for the pentahedral family, as the last 
two comparisons result in {\tt true}.
\begin{verbatim}
r5 := PolynomialRing(Integers(),5);
ff5<a,b,c,d,e> := FunctionField(Rationals(),5);
r4<x,y,z,w> := PolynomialRing(ff5,4);

lfl := [x,y,z,w,-x-y-z-w];
col := [ff5.i : i in [1..5]];
f := a*x^3 + b*y^3 + c*z^3 + d*w^3 + e*(-x-y-z-w)^3;

sy_f := [ElementarySymmetricPolynomial(r5,i) : i in [1..5]];
sigma := [Evaluate(sf,col) : sf in sy_f];

I_8  := sigma[4]^2 - 4 *sigma[3] * sigma[5];
I_16 := sigma[1] * sigma[5]^3;
I_24 := sigma[4] * sigma[5]^4;
I_32 := sigma[2] * sigma[5]^6;
I_40 := sigma[5]^8;

L_11 := sigma[5]^2 * &+[   col[i] * lfl[i] : i in [1..5]];
L_19 := sigma[5]^4 * &+[ 1/col[i] * lfl[i] : i in [1..5]]; 
L_27 := sigma[5]^5 * &+[ col[i]^2 * lfl[i] : i in [1..5]]; 
L_43 := sigma[5]^8 * &+[ col[i]^3 * lfl[i] : i in [1..5]]; 

inv := ClebschSalmonInvariants(f);
cov := LinearCovariantsOfCubicSurface(f);

inv eq [I_8, I_16, I_24, I_32, I_40];
cov eq [L_11, L_19, L_27, L_43];
\end{verbatim}
\end{proof}

\section{Performance test}
Computing the Clebsch-Salmon invariants, following the approach above, for 100 
cubic surfaces chosen at random
with two digit integer coefficients takes about 3 seconds of CPU time.
Most of the time is used for the direct evaluation of the invariant $S$ of ternary cubics by transvection. 
Note that  
computing the contravariant $\tilde{S}$ by interpolation requires 35 evaluations of the invariant $S$ of 
a ternary cubic. 
Computing both contravariants $\tilde{S}$ and $\tilde{T}$ and the dual surface takes about 18 seconds of
CPU time for the same 100 randomly chosen surfaces.

For comparison, the computation of the pentahedral form by inspecting the singular points of 
the Hessian takes about 10 seconds per example~\cite[Sec.~5.11]{EJ1}.

All computations are done on one core of an Intel i5-2400 processor running at 3.1GHz.

\end{document}